\documentclass[11pt]{article}
\usepackage{hyperref}

\usepackage{tikz-cd}
\usepackage{xcolor}
\usepackage{amsmath,amsthm,amsfonts,amssymb,amsopn,amscd,bm} 
\usepackage{color}

\def\qed{{\unskip\nobreak\hfil\penalty50
\hskip2em\hbox{}\nobreak\hfil$\square$
\parfillskip=0pt \finalhyphendemerits=0\par}\medskip}
\def\proof{\trivlist \item[\hskip \labelsep{\bf Proof.\ }]}
\def\endproof{\null\hfill\qed\endtrivlist\noindent}
\def\eproof{\null\hfill\qed\endtrivlist\noindent}

\def\tilde{\widetilde}

\def\a{\alpha}
\def\b{\beta}

\def\e{\varepsilon}

\def\l{\lambda}

\def\setminus{\smallsetminus}

\def\A{{\cal A}}

\def\E{{\cal E}}

\def\R{{\cal R}}

\def\I{{\cal I}}
\def\H{{\cal H}}
\def\K{{\cal K}}
\def\S{{\cal S}}

\def\f{{\varphi}}

\def\p{{\pi}}
\def\l{{\lambda}}

\def\PSL{{{\rm PSL}(2,\mathbb R)}}

\def\Diff{{\rm Diff}}

\def\Mob{{\rm\textsf{M\"ob}}}

\def\S2{S^{1(2)}}
\def\Poi{{\cal P}_+^\uparrow}

\def\pPoi{{\cal P}_+}

\def\RR{\mathbb R}

\newtheorem{theorem}{Theorem}[section]
\newtheorem{lemma}[theorem]{Lemma}

\newtheorem{corollary}[theorem]{Corollary}

\newtheorem{proposition}[theorem]{Proposition}

\theoremstyle{remark} 

\newcommand{\ben}{\begin{equation}}
\newcommand{\een}{\end{equation}}

\newcommand{\bthm}{\begin{theorem}}
\newcommand{\ethm}{\end{theorem}}

\newcommand{\bprop}{\begin{proposition}}
\newcommand{\eprop}{\end{proposition}}
\newcommand{\bcor}{\begin{corollary}}
\newcommand{\ecor}{\end{corollary}}
\newcommand{\blem}{\begin{lemma}}
\newcommand{\elem}{\end{lemma}}

\def\setminus{\smallsetminus}

\def\Diff{{\mathrm {Diff}}}
\def\PSL{PSU(1,1)}

\def\ZZ{{\mathbb Z}}

\def\SL2{{{\rm SL}(2,\R)}}

\def\PSL2{{{\rm PSL}(2,\Reali)}}

\def\U1{{{\rm V}(1)}}
\def\SU2{{{\rm SV}(2)}}

\def\SU{{{\rm SU}}}

\def\A{{\mathcal A}}

\def\H{{\mathcal H}}
\def\I{{\mathcal I}}

\def\K{{\mathcal K}}

\def\O{{\mathcal O}}

\def\cP{{\mathcal P}}
\def\S{{\mathcal S}}

\def\cP{{\cal P}}

\def\RR{{\mathbb R}}

\def\a{\alpha}
\def\b{\beta}
        
\def\d{\delta}        
  
\def\l{\lambda}       \def\L{\Lambda}

\def\p{\pi}

\def\f{\varphi}
\def\c{\chi}

\def\S{{S(\RR^d)}}
\def\supp{{\text{supp}\,}}

\newcommand{\hoi}{{{H}^{(1)}(I)}}
\newcommand{\hob}{{{H}^{(1)}(B)}}
\newcommand{\hki}{{{H}^{(k)}(I)}}

\newcommand{\hik}{{{H}_{(k)}(I)}}

\title{\Huge{An entropy bound due to symmetries}
} 

\author{ {\sc Roberto  Longo},  {\sc Vincenzo Morinelli}\\
Dipartimento di Matematica,
Tor Vergata Universit\`a di Roma\\
Via della Ricerca Scientifica, 1, I-00133 Roma, Italy
}

\date{}
\begin{document}

\maketitle

\begin{abstract}
Let $H: O\mapsto H(O)$ be a local net of real Hilbert subspaces of a complex Hilbert space $\H$ on the family $\O$ of double cones of the spacetime $\RR^{d+1}$, covariant with respect to a positive energy, unitary representation $U$ of the Poincaré group $\Poi$, with the Bisognano-Wichmann property for the wedge modular group. We set an upper bound on the local entropy $S_H(\phi|\! | C)$ of a vector $\phi\in \H$ in a region $C\subset \RR^{d+1}$ that depends only on $U$ and the PCT anti-unitary canonically associated with $H$. A similar result holds for local, Möbius covariant nets of standard subspaces on the circle. We compute the entropy increase and illustrate this bound for the nets associated with the $U(1)$-current derivatives. 
\end{abstract}

\newpage


\section{Introduction}
Recently, a general notion of entropy $S(\phi|\!| H)$, associated with a vector $\phi$ of a complex Hilbert space $\H$ with respect to a real linear subspace $H\subset\H$, has been provided in \cite{L19, CLR20}. Despite its broad generality and effectiveness, the full meaning of this measure of information has been uncovered only within specific contexts so far. 

We shall recall the definition of $S(\phi|\!| H)$, in terms of the modular theory of standard subspaces, and its main properties in Section \ref{S2}. For the moment, the formula
\[
S(\phi|\!| H) = S(\f_\phi |\!| \f_0)
\]
may serve as a motivation and for an initial understanding of the notion; here, {$S(\f_\phi |\!| \f_0)$} is Araki's relative entropy between the coherent state $\f_\phi$ associated with $\phi\in\H$ and the vacuum state {$\f_0$}, on the second quantization von Neumann algebra $R(H)$ relative to $H$, on the Bose Fock Hilbert space $e^\H$ over $\H$. 

The notion of $S(\phi|\!| H)$ is natural in defining the local entropy of a wave packet \cite{CLR20, LM23}, also on a curved spacetime \cite{CLRR22}, and in the context of signal transmission \cite{L23}. 
We also mention that it may describe the loss of information concerning the arrow of time \cite{BF23}. 

In this paper, we deal with local nets of real linear subspace. We start with a local, Poincaré covariant, net $H : O\mapsto H(O)$ of closed, real linear subspaces of a complex Hilbert space $\H$, see \cite{BGL02}. So, $O$ runs in the family $\O$ of double cones of the Minkowski spacetime $\RR^{d+1}$, $H(O)$ is a real Hilbert subspace of $\H$ associated with $O$, and we are given a unitary, positive energy, representation $U$ of the Poincaré group $\Poi$ acting covariantly on $H$:
\[
U(g)H(O) = H(gO), \ g \in \Poi\, , \ O\in \O
\]
(we assume $d=3$ for simplicity, most of our discussion is valid for any $d>1$);
moreover, $H$ is local:
\[
H(O_1) \subset H(O_2)'\ , \ O_1 , O_2\in \O\ ,
\]
if $O_1 , O_2$ are spacelike separated, with $K'$ denoting the symplectic complement of $K\subset \H$, namely $K' = \{\phi\in\H : \Im(\phi, \psi) = 0 \,  , \psi\in K\}$.

For example, $H$ may be associated with a local, Poincaré covariant net $\A$ of von Neumann algebras on $\H$:
\[
H(O) = \overline{\A(O)_{\rm sa}\Omega}\, ,
\]
with $\A(O)_{\rm sa}$ the selfadjoint part of $\A(O)$ and $\Omega$ the vacuum vector. But $H$ is not always of this kind. 

The real Hilbert subspace $H(C)$ can then be defined for every $C\subset \RR^{d+1}$ as the closed, real linear span generated by all the $H(O)$'s with $O\in\O$ contained in $C$ (setting $H(C) = \{0\}$ if $C$ has empty interior).

With $H$ as a local, Poincaré covariant net on $\H$, we are interested in the local entropy
\[
S(\phi|\!| C) \equiv S\big(\phi|\!| H(C)\big)
\]
of a vector $\phi\in \H$ with respect to $C\subset \RR^{d+1}$. 

We assume that $H$ is non-degenerate in a natural sense; so, by the Reeh-Schlieder theorem, $H(C)$ is cyclic and separating for every region $C\subset \RR^{d+1}$ such that both $C$ and its spacelike complement $C'$ have non-empty interiors, and we may consider the modular operator $\Delta_C = \Delta_{H(C)}$ for such a region. 

We also assume the Bisognano-Wichmann property to hold, that is
\[
 \Delta^{-is}_W = U\big(\L_W(2\pi s)\big)\, , \, s\in \RR\, ,
 \]
for any, thus for all, wedge region $W\subset \RR^{d+1}$, a property that holds in wide generality. 
Here, $\L_W$ the boost one-parameter subgroup of $\Poi$ associated with $W$. 
It follows that there exists a canonical PCT anti-unitary operator $\Theta$ on $\H$ (so $U$ extends canonically to an anti-unitary representation $\tilde U$ of the proper Poincaré group $\pPoi = \Poi \rtimes \ZZ_2$ generated by $\Poi$ and the spacetime reflection $r$ on $\RR^{d+1}$). 

We shall derive the local, universal bound for the entropy
\ben
S_H(\phi|\!| C) \leq S_{\tilde U}(\phi|\!| C)\, ,
\een
which depends only on the anti-unitary representation $\tilde U$ of the symmetry group, and not on $H$. 

An analogous set-up is associated with nets of standard subspaces on other spacetimes. We shall consider the Möbius covariant nets on the real line arising from the $U(1)$-current and its derivatives; we shall obtain explicit formulas for the local modular Hamiltonians and a vector's local entropy. The universal bound is expressed by 
\[
S\big(f|\!| H^{(k)}(-1, 1)\big) =
\pi \int_{-1}^1(1-x^2)f'(x)^2dx - \pi k(k-1)\int_{-1}^1f(x)^2dx  \leq \pi \int_{-1}^1(1-x^2)f'(x)^2dx\, .
\]
Here above, $f$ is a real function in the Schwartz space $S(\RR)$;  the right-hand side is the entropy of $f$ in the interval $(-1,1)$ w.r.t. the $U(1)$-current; the left-hand side is the entropy of $f$ in $(-1,1)$ w.r.t. the $k$-derivative of $U(1)$-current; $f$ is normalized to have zero moments $\int_{-1}^1 x^n f(x)dx$, $n= 0,1,\dots, k-2$. 
Note that the non-negativity of the latter is not obvious from this expression; however, we shall prove this directly. 

Huzihiro Araki first considered, independently, both real linear subspaces of a complex Hilbert space, with the associated von Neumann algebras in second quantization \cite{Ar63}, and the relative entropy between normal states of arbitrary von Neumann algebras \cite{Ar}. Both concepts have been very influential and lie jointly at the basis of this paper that we dedicate to him. 

\section{Entropy and standard subspaces}
\label{S2}
We recall a few facts about standard subspaces and the notion of entropy of a vector with respect to a standard 
subspace. We refer to \cite{L,L19,CLR20} for more details. 

Given a complex Hilbert space $\H$, a closed real linear subspace $H$ of $\H$ is said to be cyclic if $\overline{H+iH}=\H$, separating if $H\cap iH=\{0\}$; $H$ is cyclic iff $H'$ is separating, where $H'$ denotes the \emph{symplectic complement} of $H$ given by 
\[
H' = \big\{\psi\in\H: \Im(\psi,\phi)=0,\, \forall  \phi\in H\big\}\ .
\]
Note that $H'' = H$. $H$ is called \emph{standard} if it is both cyclic and separating. 

With a standard subspace $H$, one associates the involutive, closed, anti-linear \emph{Tomita operator} $S_H$ on $\H$, with domain $H + iH$, defined by
\[
S_H(\phi+i\psi) = \phi-i\psi\, ,\ \phi,\psi\in H\, ,
\] 
The polar decomposition $S_H = J_H \Delta_H$ gives 
an antiunitary operator $J_H$, \emph{the modular conjugation}, and a self-adjoint, positive, nonsingular 
operator $\Delta_H$, \emph{the modular operator},  satisfying the relations
$
J_H=J^*_H \ , \ \ J_H\Delta_H J_H= \Delta^{-1}_H$. 

The unitary \emph{modular group} is the one-parameter unitary group $\Delta^{is}_H$, $s\in \RR$; the crucial relations are
\[
\Delta^{is}_H H= H  \ , \ \ J_H H = H'   \, .
\]  
We say that $H$ is \emph{factorial} if $H\cap H' = \{0\}$.  Given a factorial standard  subspace $H$ of $\H$, the \emph{cutting projection} associated with $H$  is defined as
\[
P_H(\phi+\phi')= \phi \ , \qquad \phi\in H, \ \phi'\in H' \, .
\]
It turns out that $P_H$ is a densely defined, closed, real linear operator satisfying  
\[
P^2_H = P_H \  \ , \ \ P_H^* =P_{iH}  = -iP_H i\ \ , \ \  P_H\Delta^{is}_H=
\Delta^{is}_H P_H \ .
\]
$P_H$ can be expressed in terms of the modular data:
\ben\label{PD}
P_H = (1 + S_H)(1 - \Delta_H)^{-1} = a(\Delta_H) + J_H b(\Delta_H) \, ,
\een
with $a(\l) = (1 - \l)^{-1}$ and $b(\l) = \l^{1/2}a(\l)$ (note that $1$ is not in the point spectrum of $\Delta_H$ by the factoriality assumption).

With $H$ a factorial standard subspace, the \emph{entropy of a vector $\phi\in\H$ w.r.t. $H$} is defined by
\ben\label{Sf}
S(\phi |\!| H) = -\Im(\phi, P_H i\log\Delta_H \phi)\, ,
\een
in the quadratic form sense (see Sect. \ref{Evc}). 
If $H$ is not factorial, the definition \eqref{Sf} still makes sense as
\[
S(\phi |\!| H) = -i\Im\big(\phi, a(\Delta_H) \log\Delta_H \phi\big)  +i\Im\big(\phi, J_H b(\Delta_H) \log\Delta_H \phi\big) \, ,
\]
which is well defined because $a(\l)\log\l$ is not singular at $\l =1$. 

In general, if $H\subset \H$ is any closed, real linear subspace, we set $S(\phi |\!| H) = S(\phi_s |\!| H_s)$, where $H_s$ is the (standard) component of $H$ in $\H_s =(H\cap iH)^{\perp_\mathbb R}\cap \overline{H+iH}$, and $\phi_s$ is the orthogonal projection of $\phi$ on $\H_s$. 
We thus may, and will, assume the considered subspaces to be standard and factorial. 

Some of the main properties of the entropy of a vector are:
\begin{itemize}
\item $S(\phi |\!| H)\geq 0$ or $S(\phi |\!| H) = +\infty$ (\emph{positivity}); 
\item If $K \subset H$, then $S(\phi |\!| K)\leq S(\phi |\!| H)$ (\emph{monotonicity});
\item If $\phi_n  \to \phi$, then $S(\phi |\!| H) \leq \liminf_n S(\phi_n |\!| H)$ (\emph{lower semicontinuity});
\item If $H_n\subset H$ is an increasing sequence with $\overline{\bigcup_n H_n} = H$, then $S(\phi |\!| H_n) \to S(\phi |\!| H)$;
(\emph{monotone continuity});
\item $S(\phi + \psi |\!| H) = S(\phi |\!| H)$ if $\psi\in H'$, moreover $S(\psi |\!| H)= 0$ iff $\psi\in H'$ (\emph{locality});
\item $S(U\phi |\!| U H) = S(\phi |\!| H)$, $U$ unitary of $\H$ to a Hilbert space $\K$ (\emph{unitary invariance}). 
\end{itemize}
Here, $\phi, \phi_n, \psi\in\H$ and $H,H_n$ are closed, real linear subspaces. The locality property will be proved in Prop. \ref{locality}. 

\section{Nets of standard subspaces}
Let $\H$ be a complex Hilbert space, and $\O$ the family of double cones of the Minkowski spacetime $\RR^{d+1}$. A local Poincaré covariant {\it net of real linear subspaces} of $\H$ is a map
\[
O\in \O \mapsto H(O)\subset \H \, ,
\]
with $H(O)$ a real linear subspace of $\H$, that we assume to be closed (we may otherwise replace $H(O)$ by its closure),
that satisfies the following:
\begin{itemize}
\item $O_1 \subset O_2 \implies H(O_1) \subset H(O_2)$, $O_1 , O_2\in\O$ (\emph{isotony}); 
\item $O_1 \subset O'_2 \implies H(O_1) \subset H(O_2)'$, $O_1 , O_2\in\O$ (\emph{locality});
\item There exists a unitary, positive energy representation $U$ of $\Poi$ on $\H$ such that $U(g)H(O) = H(gO)$, $O\in\O$, $g\in\Poi$ (\emph{Poincaré covariance});
\item For every $O\in\O $, $\overline{\sum_{x\in\RR^{d+1}}H(O+x)}= \H$ (\emph{non-degeneracy}). 
\end{itemize}
As said, given a net $H$ as above, we define $H(C)\subset \H$ for every region $C$ as
the closed, real linear span of all $H(O)$'s with $O\subset C$. 
\bthm \emph{(Reeh-Schlieder theorem).}
$H(C)$ is cyclic for every $C\subset\RR^{d+1}$ with non-empty interior. Therefore, $H(C)$ is standard (i.e., cyclic and separating) if both $C$ and its spacelike complement $C'$ have a non-empty interior. 
\ethm
\proof
The argument is classical. 
Let $C$ have a non-empty interior, so there exists a double cone $O\in\O $ such that $O + x \subset C$ for $x\in\RR^{d+1}$ in a neighbourhood $\cal U$ of $0$. If $\phi\in H(C)'$, then $\phi\in H(O +x)'$, $x\in\cal U$; so 
$F(x) \equiv\Im(\phi, U(x)\psi) = 0$ for every $\psi \in H(O)$, $x\in \cal U$. By the positivity of the energy, $F$ is the boundary value of a function analytic in 
$\RR^{d+1} + i V_+$, with $V_+$ the forward light cone; so $F$ is identically zero as it vanishes on $\cal U$. We conclude that $\psi$ is in the symplectic complement of $\sum_{x\in \RR^{d+1}}H(O+ x)$, so $\psi =0$, because $H$ is assumed to be non-degenerate. 
\endproof
By the Reeh-Schlieder theorem, given $C$ such that both $C$ and $C'$ have a non-empty interior, we may consider the modular operator and the modular conjugation
\[
\Delta_C = \Delta_{H(C)}\, ,\ \quad J_C = J_{H(C)}\, .
\]
The following property plays a crucial role:
\begin{itemize}
\item For every wedge region $W\subset\RR^{d+1}$,
\ben\label{BW}
\Delta^{-is}_W = U\big(\L_W(2\pi s)\big)\, , \ s\in\RR\, ,
\een
\emph{(Bisognano-Wichmann property)}. 
\end{itemize} 
Here, $\L_W$ is the one-parameter boost subgroup of $\Poi$ leavings $W$ globally invariant. By Poincaré covariance, the Bisognano-Wichmann property holds for all wedges if it holds for one given wedge. 
We mention that the Bisognano-Wichmann property holds automatically in some instances. In particular,
if $H$ is a local Poincaré covariant, net of real linear spaces, and the unitary representation $U$ is
a direct integral of any scalar representations, then the Bisognano-Wichmann property holds for $H$, see \cite[Thm. 4.4]{M16}. 
\bthm\label{pP}
Let $H$ be a net as above with the Bisognano-Wichmann property. Then:
\begin{itemize}
\item[$(i)$] The representation $U$ of $\Poi$ is unique;
\item[$(ii)$]  For every wedge $W$,
\[
H(W)' = H(W') 
\]
(\emph{Wedge duality});
\item[$(iii)$]  If $U$ does not contain the identity representation, then $H(W)$ is factorial for every wedge $W$, namely $H(W)\cap H(W)' = \{0\}$;
\item[$(iv)$]The anti-unitary $\Theta$ given by $J_W U(R_W)$ is independent on the wedge $W$, with zero in the edge of $W$;  
$\Theta$ is an involution and satisfies 
\[
\Theta H(O) = H(-O)\, ,\quad O\in\O\, ,\quad {\rm and}\quad \Theta U(g) \Theta = U(rgr)\, , \ g\in\Poi\, ,
\]
(\emph{PCT}). Here, $R_W$ is the space $\pi$-rotation preserving $W$ and $r$ is the spacetime reflection $r: x\mapsto -x$. 

Therefore, $J_W = \Theta U(R_W)$ acts geometrically on $H$ as the reflection on $\RR^{d+1}$ around the edge of $W$. 
\end{itemize}
\ethm
\proof
$(i)$: The condition \eqref{BW} fixes the restrictions of $U$ to $\L_W$; since the group generated by the boost subgroups $\L_W$'is $\Poi$, the representation $U$ is fixed. 

$(ii)$: If $K$ is a standard subspace, then $\Delta_{K'} = \Delta_K^{-1}$, so
the standard subspaces $H(W')'\supset H(W)$ share the same modular unitary group $U\big(\L_W(-2\pi t)\big)$, so they coincide \cite{L}. 

$(iii)$: If $K$ is a standard subspace, the center $K\cap K'$ of $K$ is equal to the fixed-vector subspace for the modular group $\Delta_K^{it}$. By the Bisognano-Wichmann property,  the center of $H(W)$ is equal to the $U\big(\L_W(\cdot)\big)$-fixed points. Assuming that $U$ has no non-zero fixed vector, it follows that $U\big(\L_W(\cdot)\big)$ has no non-zero fixed vector, so $H(W)$ is factorial. 

$(iv)$: This follows along the same lines of \cite{BGL94} and \cite[Thm. 2.10]{GL95}.
\eproof
Let $\pPoi$ denote the proper Poincaré group; $\pPoi$ is generated by $\Poi$ and the spacetime inversion $r$, indeed $\pPoi$ is the semi-direct 
product $\Poi \rtimes \ZZ_2$, with $r$ the non-identity element of $\ZZ_2$.  
An anti-unitary representation $\tilde U$ of $\pPoi$ is a representation of $\pPoi$ by either unitary or anti-unitary operators, 
such that $\tilde U|_{\Poi}$ is unitary and $\tilde U(r)$ is anti-unitary. 

$(iv)$ of Thm. \ref{pP} can be rephrased as follows.
\bcor
Let $H$ be a net as above. The unitary representation $U$ of $\Poi$ extends canonically to an anti-unitary representation $\tilde U$ of $\pPoi$ acting covariantly on $H$. 
\ecor
\proof
Theorem 2.3 in \cite{BGL94} proves a similar statement (with the weaker assumption of modular covariance) in the setting of local, Poincaré covariant nets of von Neumann algebras. Our corollary can be proved following the same lines. 
\eproof

\subsection{The dual net}
Let $H$ be a local Poincaré covariant, net of real linear spaces with the Bisognano-Wichmann property. We define the \emph{dual net}
$H^d$ by setting
\ben
H^d(O) = H(O')'\, ,\quad O\in\O\, ,
\een
(intersection over all wedges containing $O$). 

By locality, $H(O')\subset H(O)'$, therefore
\ben
H(O) \subset H^d(O)\, ,\quad O\in\O\,  .
\een
\bprop
$H^d$ is a local, Poincaré covariant net of real linear spaces with the Bisognano-Wichmann property.
\eprop
\proof
The isotony of $H^d$ is clear: if $O_1\subset O_2$, then  $O'_1\supset O'_2$, so
$H(O'_1)\supset H(O'_2)$ because $H$ is isotonous; thus $H^d(O_1)\subset H^d(O_2)$ as taking the symplectic complement reverses the inclusion order. 

Note now that
 \ben\label{w1}
H^d(W) = H(W)\, ,\quad W \ {\rm wedge}\, ,
\een 
as $H(W')' = H(W)$ by wedge duality. Therefore
 \ben\label{w3}
H^d(O) = \bigcap_{W\supset O} H(W)\, ,\quad O\in\K\, ,
\een
($W$ wedge); indeed $H^d(O) = H(O')' = \left(\sum_{W\subset O'} H(W)\right)' = \bigcap_{W\supset O} H(W)$. 

If the double cones $O_1, O_2$ are spacelike separated, then there exists a wedge $W$ such that $O_1\subset W$, $O_2\subset W'$; thus 
\[
H^d(O_1) \subset H(W) = H(W')' \subset H^d(O_2)'
\]
and locality holds for $H^d$. 

The Poincaré covariance, non-degeneracy, and Bisognano-Wichmann properties for $H^d$ follow immediately from the corresponding properties of $H$ and \eqref{w1}. 
\endproof
We shall say that the net $H$ satisfies \emph{Haag duality} if
\[
H(O)' = H(O')\, ,\quad O\in \O \, .
\]
\bprop
The dual net $H^d$ satisfies Haag duality. Therefore, $H = H^d$ iff $H$ satisfies Haag duality.  
\eprop
\proof
Note that, if $W$ is a wedge and $O\subset W$ is a double cone, then $H^d(W) = \overline{\sum_{O\subset W} H^d(O)}$, because $H(O)\subset H^d(O) \subset H(W)$. 

Fix now $O\in\O $. By \eqref{w3}, we have
\[
H^d(O)'  = \overline{\sum_{W\subset O'} H(W)} = \overline{\sum_{W\subset O'} H^d(W)} = \overline{\sum_{O_1\subset O'} H^d(O_1)} = H^d(O')
\]
($W$ wedge, $O_1$ double cone), so $H^d$ is Haag dual. The rest is clear. 
\eproof
It follows that $H^d$ is the maximal extension of $H$ on $\H$ that is relatively local with respect to $H$. 

The following theorem follows from \cite[Thm 4.7]{BGL02} and \cite[Thm. 6.1]{LMR16}. 
\bthm\label{UP}
With $\H$ a complex Hilbert space, there exists a one-to-one correspondence between:
\begin{itemize}
\item[$(a)$] Anti-unitary, positive energy, representations of $\tilde U$ of $\pPoi$ on $\H$ such that $U = \tilde U|_{\Poi}$ does not contain infinite spin subrepresentations;
\item[$(b)$] Poincaré covariant, Haag dual nets $H$ of real linear subspaces on $\H$ with the Bisognano-Wichmann property \eqref{BW}. 
\end{itemize}
\ethm
\proof
Given a net in $(b)$, the covariance unitary representation $U$ of $\Poi$, the anti-unitary representations of $\tilde U$ of $\pPoi$ is generated by $U$ and the modular conjugation $J_W$ of any wedge $W$ whose edge contains the origin \cite{BGL94}. $U$ does not contain infinite spin subrepresentations because the local real Hilbert subspaces $H(O)$, $O\in\O$, are cyclic \cite{LMR16}. 

Conversely, given an anti-unitary representation $\tilde U$ as in $(a)$, the standard subspace $H(W)$ associated with a wedge $W$ is defined as the fixed-point real linear space of the operator $S_W = J_W\Delta_W^{1/2}$, where $\Delta_W$ is defined by eq. \eqref{BW} 
and $J_W$ by $(iii)$ of Thm. \ref{pP} \cite{GL95}. Then $H(O)$ is defined by \eqref{w3}. 
\eproof
\bcor\label{UD}
Let $H$ be a local, Poincaré covariant net real linear subspaces on $\H$ with the Bisognano-Wichmann property and $U$ the unitary covariance representation of $\Poi$. 

Then the dual net $H^d$ depends only on the anti-unitary $\tilde U$ of $\pPoi$ and not on $H$. 
\ecor
Indeed, $H^d$ is the local net corresponding to the unitary representation $\tilde U$ of $\pPoi$ in Theorem \ref{UP}. 

Let $\A$ be a local, Poincaré covariant net of von Neumann algebras on a Hilbert space, see \cite{ArBook}.  We  assume that the Bisognano-Wichmann property holds and that $\A$ satisfies Haag duality
\[
\A(O) = \A(O')'\, , \quad O\in \O\, .
\]
Let $H$ be the associated net of standard subspaces. 
\[
H_\A(O) = \overline{\A(O)_{\rm sa}\Omega}\, , \quad O\in \O\, ,
\]
with $\A(O)_{\rm sa}$ the selfadjoint part of $\A(O)$ and $\Omega$ the vacuum vector. In general (possibly always), $H_\A$ is not a dual net, namely
\ben\label{sincl}
H_\A(O) \nsubseteq H_\A^d(O)\, ,  \quad O\in \O\,  ,
\een
namely the inclusion $H_\A(O) \subset H_\A^d(O)$ is strict if $O$ is a double cone (obviously, $H_\A(W) = H_\A^d(W)$ is $W$ is a wedge). 
We shall discuss the analogous statement for chiral conformal nets in Section \ref{sectf}. 

{We remark that Theorem \ref{UP} and Corollary \ref{UD} continue to hold when nets of standard subspaces with spacelike cone localization and infinite spin representations are taken into account. Indeed, in \cite{LMR16}, it is proved that if $\tilde U|_{\cP_+^\uparrow}$ contains infinite spin representations, then $H^d(O)$ can not be cyclic and $H^d(O)$ is trivial if all subrepresentations of $U$ have infinite spin. On the other hand, $H^d(C)$ is always cyclic if $C$ is a spacelike cone \cite{BGL02}. In this more general setting, Haag duality is the equality $H^d(C)=H^d(C')'$, for all spacelike cones $C$.}

\section{Universal bound}
Let $H$ be a local, Poincaré covariant net of real linear subspaces on the complex Hilbert space $\H$, with covariance unitary representation $U$.  Given $\phi\in \H$, we are interested in a bound for the \emph{entropy of $\phi$ with respect the region $C\subset \RR^{d+1}$ relative to $H$} defined by
\[
S_H(\phi|\!| C) \equiv S\big(\phi|\!| H(C)\big)\, .
\]
Given an anti-unitary, positive energy representation $V$ of $\pPoi$,
 we then define \emph{the entropy of $\phi$ with respect to $C$ associated with $V$} as
 \[
 S_{V}(\phi|\!| C) \equiv S_K(\phi|\!| C)\, ,
 \]
where $K$ is the local net of real linear spaces associated with $V$ by Theorem \ref{UP}. 
\bthm Let $H$ be a local, Poincaré covariant net real linear subspaces on the complex Hilbert space $\H$, with the Bisognano-Wichmann property. 
With $U$ the covariance unitary representation of $\Poi$, let $\tilde U$ be the canonical extension of $U$ to an anti-unitary representation of $\pPoi$. 

For every region $C\subset \RR^{d+1}$ and vector $\phi\in\H$, the bound
\ben\label{UB}
S_H(\phi|\!| C) \leq S_{\tilde U}(\phi|\!| C)
\een
holds and depends only on $\tilde U$, not on $H$. 
\ethm
\proof
Since $H(C)\subset H^d(C)$, we have 
\[
S_H(\phi|\!| C) \leq S_{H^d}(\phi|\!| C) \, ;
\]
by the monotonicity of the entropy. The inequality \eqref{UB} thus holds because $H^d$ is the net associated with $\tilde U$, so
\[
S_{H^d}(\phi|\!| C) = S_{\tilde U}(\phi|\!| C)\,  .
\]
\eproof
Let $V$ be a unitary, positive energy anti-representation of $\pPoi$  on the Hilbert space $\H$, with finite spin. {Then $V$ has a direct integral decomposition
\[
V = \int_X^\oplus V_\l d\mu_\l 
\]
on $\H = \int_X^\oplus \H_\l d\mu_\l$,
where  $V_\l$ is irreducible. Let $V_\l |_{\Poi}$ be the restriction of $V_\l$ to $\Poi$. Then either $V_\l |_{\Poi}$ is irreducible and self-conjugate, or $V_\l |_{\Poi} = U_\l\oplus \bar U_\l$, with $U_\l$ irreducible and not self-conjugate.

As is well known, $V_\l |_{\Poi}$ may have mass $m>0$ and arbitrary spin $s\in\frac12  \mathbb Z$, or $m = 0$ scalar (case of irreducible $V_\l |_{\Poi}$), or  $V_\l |_{\Poi}= U_\l\oplus \bar U_\l$, with $U_\l$ massless with non-zero helicity. Infinite spin representations do not appear here because the Reeh-Schlieder property for double cones holds \cite{LMR16}.

Every vector $\phi\in\H$ decomposes as $\phi =  \int_X^\oplus \phi_\l d\mu_\l$ and we have an integral expression for the entropy
\[
S_V(\phi|\!| C) = \int_X S_{V_\l}(\phi_\l |\!| C)d\mu_\l \, .
\]}

\section{Möbius covariant nets}
The previous discussion in the context of Poincaré covariant nets can be similarly done for nets on different spacetimes. We consider nets on the circle, that are covariant w.r.t. a unitary representation of the Möbius group $\Mob$, and we then explicitly compute the universal entropy bound within this framework. 

Let $\H$ be a complex Hilbert space, and $\I$ the family of proper intervals of $S^1$. A (local) \emph{Möbius covariant net} of $\H$ is a map
\[
I\in \I \mapsto H(I)\subset \H
\]
with $H(I)$ a real, closed linear subspace of $\H$, that satisfies the following
\begin{itemize}
\item $I_1 \subset  I_2 \implies H(I_1) \subset H(I_2)$ (\emph{isotony}); 
\item $I_1 \subset I'_2 \implies H(I_1) \subset H(I_2)'$ (\emph{locality});
\item There exists a unitary, positive energy representation $U$ of $\Mob$ on $\H$ such that $U(g)H(I) = H(gI)$ (\emph{Möbius covariance});
\item $\overline{\sum_{I\in\I} H(I)}= \H$ (\emph{non-degeneracy}). 
\end{itemize}
Here, $I' = S^1\setminus I$. In this case, the Reeh-Schlieder property, the uniqueness of $U$, and the Bisognano-Wichmann property 
\[
U\big(\delta_I(-2\pi s)\big)=\Delta_H^{is}
\]
are automatic, where $\delta_I$ is the one-parameter group of ``dilations'' associated with $I$. 

Haag duality holds:
\[
H(I') = H(I)'\, , \ I\in \I\, .
\]
The proof is analogous to the one in \cite{GL96}. 

If $U$ has no non-zero fixed vector, then $H(I)$ is factorial, $I\in\I$, as in Theorem \ref{pP}, thus
\ben\label{irr0}
H(\RR)  \supset \overline{H(-\infty, 0) + H(0, +\infty)} = \H
\een
because 
\[
 \big(H(-\infty, 0) + H(0, +\infty)\big)' = H(-\infty, 0)'\cap H(0, +\infty)' = H(0, +\infty)\cap H(0, +\infty)' = \{0\}\, .
\]
We may also define a (local) translation-dilation covariant net on the real line $\RR$
\[
I \in \I_0 \mapsto H(I)\subset \H\, ,
\]
with $\I_0$ the family of bounded intervals of $\RR$; we assume the Reeh-Schlieder property for $I\in\I_0$ and the Bisognano-Wichmann property for half-lines. 
All translation-dilation covariant nets on $\RR$ arise uniquely by restricting a Möbius covariant net on $S^1$ \cite{GLW98}. 

Given a translation-dilation covariant net $H$ on $\RR$, the dual net is defined by
\[
H^d(I) = H(\RR\setminus I)'\, , \ I\in \I_0\, .
\]
\bthm Let $H$ be a local, Möbius covariant net of closed real linear subspaces on the complex Hilbert space $\H$ on $S^1$. Let $H^d$ be the dual net, and $\tilde U$ the {covariance} unitary representation of {\rm $\Mob$} associated with $H^d$.  

For every interval $I\in\I_0$ of the real line, the bound
\ben\label{UBI}
S_H(\phi|\!| I) \leq S_{\tilde U}(\phi|\!| I)
\een
holds and depends only on $\tilde U$, not on $H$. 

$\tilde U$ is quasi-equivalent to the positive energy unitary representation of {\rm $\Mob$} with lowest weight $1$. 
\ethm
 
\section{Entropy, $U(1)$-current and its derivatives}
In this section, we illustrate the universal bound in the Möbius covariant case, by explicitly computing for in the one-particle net of standard subspaces associated with the $k$-derivative of the $U(1)$-current. 
\subsection{Nets associated with the $U(1)$-current and its derivatives}
 We sketch here the underlying structure, see also  \cite{GLW98}.
Denote by {$\frak L$} 
the real linear space of real, tempered distributions $f\in S(\RR)'$ such that the Fourier transform $\hat f$ is a measurable function;  with $k$ a positive integer, set
\[
||f ||^2_k \equiv 2\int_{\RR_+} p^{2k-1}|\hat f(p)|^2 dp < \infty\, .
\]
(here, the Fourier transform is normalized to be isometric). Then, $||\cdot||_k$ is a pseudo-nom on  $X^{(k)}\equiv  {\frak L}  + \RR^{2(k-1)}[x]$ that degenerates on $\RR^{2(k-1)}[x]$, where $\RR^n[x]$ are the real polynomials of degree less or equal to  $n$.

The quotient $\H^{(k)} \equiv X^{(k)}/\RR^{2(k-1)}[x]$ is a Hilbert space with a real scalar product
\[
\langle f,g\rangle_k= 2\int_{\RR_+} p^{2k-1}\hat f(-p)\hat g(p)dp \, .
\]
The class of $f\in X^{(k)}$ in $\H^{(k)}$ is denoted by $[f]_k$. For notational simplicity, we may write $f$ instead of $[f]_k$ if no confusion arises. 
Clearly, the Schwartz space $S(\RR)$ faithfully embeds into $\H^{(k)}$ with dense range. {$\H^{(k)}$ is equipped with a complex structure $\iota_k: X^{(k)}\to X^{(k)}$, i.e. $\iota_k$ is isometric and $\iota_k^2 = -1$ ($\widehat{\iota_k f}(p) = i\,{\rm sign}(p)\hat f(p)$, $f\in S(\RR)$, independently of $k$),}  and
the symplectic form $\beta_k$ gives the imaginary part of the scalar product on $X^{(k)}$
\ben\label{symp}
\beta_k(f,g)=\int_\RR f^{(k-1)}(x)g^{(k)}(x)dx=-\int_\RR f^{(k)}(x)g^{(k-1)}(x)dx\, , \ f,g \in X^{(k)}\, .
\een 
{($h^{(k)}$ denotes the $k$-derivative of the function $h$.)}

The complex scalar product, and the norm, on $X^{(k)}/\RR^{2(k-1)}[x]$ are thus given by
\[
( f,g)_k=\langle f,g\rangle_k+i \beta_k(f,g),\qquad (f,f)_k =\|f\|_k^2\, , \ f \in X^{(k)}\, .
\]
$\H^{(k)}$ carries the irreducible, positive energy, unitary representation $U^{(k)}$ of $\mathrm{PSL}(2,\RR)$ with lowest weight $k$ given by
\[
(U^{(k)}(g)f)(x)=(cx-a)^{2(k-1)}f(g^{-1}x)\, , \ f\in X^{(k)}\, ,
\]
$g=\left(\begin{array}{cc} a&b\\c&d
\end{array}\right)\in\mathrm{SL}(2,\RR)$, $gx=\frac{ax+b}{cx+d}$. 

The Möbius covariant net $H^{(k)}$ of standard subspaces of $\H^{(k)}$ associated with $U^{(k)}$ is given by
\[
\hki=\big\{[f]_k \in \H^{(k)}:   f \in X^{(k)}, \,\supp\,f\subset I\big\}\, ,\ I\in\I_0\, ,
\]
Now, the $(k-1)$-derivative gives a well defined one-to-one linear map
{
\ben\label{idif}
D^{k-1}: \H^{(k)}\to \H^{(1)}\, ,\quad
D^{k-1}[f]_k = [f^{(k-1)}]_1\, ,\ f\in X^{(k)}\,   .
\een}
$D^{k-1}$ is isometric, thus gives a unitary, that identifies the Hilbert spaces $\H^{(k)}$ and $\H^{(1)}$,  
for $k\geq2$. ${D^{k-1}}$  carries $H^{(k)}(I)$  onto a subspace $\hik\equiv D^{k-1}H^{(k)}(I)$ of $H^{(1)}(I)$, $I\in \I_0$.  $H_{(k)}$ is the subnet of $H^{(1)}$ on $\H^{(1)}$ given by
\[
\hik=\Big\{g \in S(\RR) : \int_\RR x^{n} g(x)dx=0, n = 0,1,\ldots,k-2,\,\supp g\subset I \Big\}^- \subset\hoi  
\]
(closure on $\H^{(1)}$; ${H_{(1)}(I)} = H^{(1)}(I)$). 
We denote by $B$ the unit interval $B=(-1,1)$. The M\"obius one-parameter group of dilations associated with $B$ is
\[
\delta_{B}(s)x=\frac{\cosh(s/2)x +\sinh(s/2)}{\sinh(s/2)x+\cosh(s/2)}\,  ,
\]
The Bisognano-Wichmann property holds for the net $H^{(k)}$:
\[
U^{(k)}\big(\delta_B(-2\pi s)\big)=\Delta_{H^{(k)}(B)}^{is}\, ,
\]
where
\ben\label{Uk}
\Big(U^{(k)}\big(\delta_B(-2\pi s)\big)f\Big)(x) = \big(\sinh(\pi s)x+\cosh(\pi s)\big)^{2(k-1)}f\big(\delta_B(2\pi s)x\big)\, .
\een
\bprop
The dual net $H_{(k)}^d$ of $H_{(k)}$ on $\H^{(1)}$ is $H^{(1)}$. 
\eprop
\proof
$H^{(1)}$ is strongly additive; namely Haag duality holds on the real line \cite{BSM90}. $H_{(k)}$ is a subnet of $H^{(1)}$ sharing the same translation-dilation representation unitary representation and the Bisognano-Wichmann property holds, see \cite{L}. Thus, $H^{(1)}$ and $H_{(k)}$ are equal on half-lines.  
Therefore, 
\[
H^{(1)}(I)'=H^{(1)}(\RR\setminus I)' =\big(H^{(1)}(I_1) + H^{(1)}(I_2)\big)' =  \big(H_{(k)}(I_1) + H_{(k)}(I_2)\big)' = H^d_{(k)}(I)\, ,
\]
if $I\in\I_0$,
where $I_1, I_2$ are the two half-lines in $\RR\setminus I$. The proposition thus holds. 
\eproof
As immediate consequence, the universal bound for the net $H_{(k)}$ takes the form
\[
S\big(f|\!| H_{(k)}(I)\big)\leq S\big(f|\!| H^{(1)}(I)\big) = S_{U^{(1)}}(\phi|\!| I)\, ,
\]
for every set $I\subset \RR$. 
We now compute these quantities. 

\subsection{Local entropy in the $U(1)$-current model}
Formula \eqref{Uk} in case $k=1$ specifies to $\Big(U^{(1)}\big(\delta_B(-2\pi s)\big)f\Big)(x)=f\big(\delta_B(2\pi s)x\big)$. 
Denote by $C^\infty_0(\RR)$ the space of real, smooth functions on $\RR$ with compact support. 
\bprop\label{MH1}
The modular Hamiltonian associated with $H^{(1)}(B)$ on $\H^{(1)}$ is given by
\[
{\iota_1\log\Delta_{H^{(1)}(B)} f = {\pi (1-x^2)f'}\, ,\ f\in C^\infty_0(\RR)\, };
\]
$C^\infty_0(\RR)$ is a core for 
$\log\Delta_{H^{(1)}(B)}$. 
\eprop
\proof
With $f\in C^\infty_0(\RR)$, we have (as in Prop. \ref{propk})
\begin{align*}
{\iota_1(\log\Delta_{H^{(1)}(B)} f)(x)}&=\frac d{{ds}}U^{(1)}\big(\delta_B(-2\pi s)f\big)\big|_{s=0}(x) \\
& =\frac d{{ds}} f\big(\delta_B(2\pi s)x\big)  \big|_{s=0}  ={ \pi (1-x^2)f'(x)} \, .
\end{align*}
$C^\infty_0(\RR)$ is a dense, $U^{(1)}$-globally invariant subspace in the domain of the generator $\log\Delta_{H^{(1)}(B)}$, thus a core for 
$\log\Delta_{H^{(1)}(B)}$. 
\eproof
{The above formula for the modular Hamiltonian and the following corollary are the one-dimensional versions of the ones in 
\cite{LM23}. See \cite{GP} for a specific derivation.}
\begin{corollary}\label{SH1}
Let $f\in S(\RR)$, so $[f]_1\in \H^{(1)}$. We have
\[
S\big(f|\!| H^{(1)}(B)\big)={\pi}\int_B(1-x^2)f'(x)^2dx\, .
\]
\end{corollary}
\proof
By locality, we may assume that $f\in C_0^\infty(\RR)$. 
With $\chi_B$ the characteristic function of $B$,
the function ${h= \pi\chi_B  (1-x^2)f'}$, and its derivative, belong to $L^2(\RR)$, so we have a well defined element $[h]_1\in \H^{(1)}$, indeed 
$[h]_1\in H^{(1)}(B)$. As in Lemma \ref{Pchi}, 
\[
P_{H^{(1)}(B)}[h]_1 = [\chi_B h]_1\, ,
\]
with $P_{H^{(1)}(B)}$ the cutting projection of $\H^{(1)}$ onto $H^{(1)}(B)$; so, by Proposition \ref{MH1}, we have
\begin{align*}
S\big(f|\!| H^{(1)}(B)\big) &= {-\beta_1\big\langle f, P_{H^{(1)}(B)}\iota_1\log\Delta_{H^{(1)}(B)}f\big\rangle }\\
&{=  \int_\RR f'(x)\chi_B(x) \big(\iota_1\log\Delta_{H^{(1)}(B)}f\big)(x)dx }\\
&= \pi \int_\RR f'(x)\chi_B(x) (1-x^2)f'(x)dx \\
&=\pi\int_B (1-x^2)f'(x)^2dx\, .
\end{align*}
\eproof
\subsection{Local entropy for derivatives of the $U(1)$-current}
Extending the previous discussion, we now consider the nets associated with the $U(1)$-current derivatives.  
\bprop\label{propk}
Let $f\in C^\infty_0(\RR)$. Then $f$ belongs to the domain of the modular Hamiltonian $\log\Delta_{H^{(k)}(B)}$ on $\H^{(k)}$ associated with $H^{(k)}(B)$ and
\ben\label{LDk}
\big(\iota_k\log\Delta_{H^{(k)}(B)}f\big)(x)
=2\pi\big(k-1)xf(x)+ \pi(1-x^2)f'(x) \, ,
\een
The space $C^\infty_0(\RR)$ is core for $\log\Delta_{H^{(k)}(B)}$. 

More generally, \eqref{LDk} holds for every real, smooth function $f$ on $\RR$ such that $f, f'\in X^{(k)}$.  
\eprop
\proof
Set $f_s (x)$ equal to the right-hand side of \eqref{Uk}. Since $f_s\in C^\infty_0(\RR)$, the incremental ratio $(f_s - f)/s$ tends uniformly to $\frac{d}{ds} f_s|_{s=0}$ as $s\to 0$ (by Lagrange theorem), and a similar statement holds for higher order $x$-derivatives of $f$. Therefore,
\begin{align*}
\b_k\big(\lim_{s\to 0} (f_s- f)/s , g\big) &= {{(-1)^{k}}} \int_\RR g^{(2k-1)}(x) \Big(\lim_{s\to 0} \big(f_s(x)- f(x)\big)/s\Big) dx \\
 &= {(-1)^{k}} \int_\RR g^{(2k-1)}(x)\frac{d}{ds} f_s(x)\Big |_{s=0} dx \, .
\end{align*}
Note that the closure $H^{(k)}(\RR)$ of $C^\infty_0(\RR)$ is equal to $\H^{(k)}$. 
As $(f_s - f)/s$ is uniformly bounded in the $||\cdot||_k$-norm as 
$s\to 0$, we conclude that $(f_s - f)/s\to \frac{d}{ds} f_s\big |_{s=0}$ 
weakly in $\H^{(k)}$, $s\to 0$. Indeed,  $(f_s - f)/s\to \frac{d}{ds} f_s\big |_{s=0}$ in the $||\cdot||_k$-norm, because $U^{(k)}\big(\delta_B(-2\pi \cdot)\big)$ is a one-parameter unitary group. So, $\frac{d}{ds} f_s\big |_{s=0}$ belongs to the domain of the generator of $U^{(k)}\big(\delta_B(-2\pi \cdot)\big)$. 

We then have:
\begin{align*}
\big(\iota_k\log\Delta_{H^{(k)}(B)}f\big)(x)& = \Big(\frac d{ds}U^{(k)} \big(\delta_B(-2\pi s)f\big)\Big|_{s=0}\Big)(x) \\
&= \lim_{s\to 0} \big(f_s(x) - f(x)\big)/s\\
&= 2\pi(k-1) xf(x)+\pi(1-x^2)f'(x)\, .
\end{align*}
The core property follows as in Prop. \ref{MH1}. 
\eproof
Let $g\in S(\RR)$ with $g^{(n)}(\pm 1) = 0$, $n= 0, 1,\dots, k-1$, then $h \equiv\chi_B g$ is of class $C^{(k-1)}(\RR)$, the $k$-derivative $h^{(k)}(x)$ exists if $x\neq \pm1$, and $h^{(k)}\in L^2(\RR)$. It follows that $h$ gives a vector $[h]_k$ in $\H^{(k)}$, similarly as in the previous section. 
\blem\label{Pchi}
Let $P_{H^{(k)}(B)}$ be the cutting projection of $\H^{(k)}$ onto $H^{(k)}(B)$ and $f\in X^{(k)}$ a real, smooth function. Then
\[
P_{H^{(k)}(B)} [f]_k = [\chi_B g]_k\, , 
\]
if there exists
$g\in X^{(k)}$, $[g]_k = [f]_k$, and $g^{(n)}(\pm 1) = 0$, $n= 0, 1,\dots, k-1$ (thus $[\chi _Bg]_k\in H^{(k)}(B)$).  
\elem
\proof
Clearly, 
\ben\label{PHd}
[f]_k = [g]_k = [\chi_B g]_k + \big([g]_k - [\chi_B g]_k\big) \, ;
\een
now, $[\chi _Bg]\in H^{(k)}(B)$,
and $\big([g]_k - [\chi_B g]_k\big)\in H^{(k)}(B')$; therefore \eqref{PHd} follows from the definition of the cutting projection. 
\eproof
We note now the identity
\ben\label{IdD}
\big(2(k-1)xf(x) +(1-x^2)f'(x)\big)^{(k-1)}
= k(k-1) f^{(k-2)}(x)+ (1-x^2)f^{(k)}(x)\, ,
\een
$f\in S(\RR)$,
that can be checked by induction. 
\begin{theorem}\label{thm:ent}
With {$k\geq 1$ an integer and} $f\in S(\RR)$, thus $[f]_k\in \H^{(k)}$,
we have 
\ben\label{SkB}
S\big([f]_k |\! | H^{(k)}(B)\big) = \pi\int_{B}(1-x^2)g^{(k)}(x)^2dx - \pi  k(k-1) \int_{B}g^{(k-1)}(x)^2dx  \, ,
\een
where $g\in X^{(k)}$, $[g]_k = [f]_k$, and $g^{(n)}(\pm 1) = 0$, $n = 0, 1, \dots, k- {2}$. 
\end{theorem}
\begin{proof}
By Lemma \ref{Pchi}, and using the identity \eqref{IdD}, we have
\begin{align*}
S\big([f]_k |\! | H^{(k)}(B)\big)&={-\beta_k}\big\langle g, P_{H^{(k)}(B)} \iota_k\log\Delta_{H^{(k)}(B)}g\big\rangle_k\\
&={\int_\RR g^{(k)}\big(\chi_B \iota_k\log\Delta_{H^{(k)}(B)}g\big)^{(k-1)}dx}\\
&=\pi\int_B g^{(k)}(x)\Big(k(k-1)g^{(k-2)}(x)+ (1-x^2)g^{(k)}(x)\Big)dx\\
 &=-\pi \int_B k(k-1)g^{(k-1)}(x)^2 dx+ \pi \int_B(1-x^2)g^{(k)}(x)^2dx \, .
 \end{align*} 
\end{proof}
Recall the unitary identification given by \eqref{idif} and take $f\in S(\RR)$, thus $[f]_1\in\H^{(1)}$ and $(D^{k-1})^*[f]_1 \in \H^{(k)}$. 
We have $(D^{k-1})^*[f]_1 = [g]_k$ with $g\in X^{(k)}$ a real, smooth function such that $g^{(k-1)} = f$.
Since $D^{k-1}$ is a unitary, then 
\[
S\big([f]_1 |\!| H_{(k)}(B)\big) = S\big((D^{k-1})^* [f]_1 |\!| (D^{k-1})^* H^{(k)}(B)\big) = S\big([g]_k |\!| H^{(k)}(B)\big)
\]
and the following corollary holds. 
\begin{corollary}\label{cor:ent}
Let $f\in S(\RR)$ with $\int_B x^n f(x)dx = 0$, $n = 0, 1, \dots, k-2$. 
The entropy of $[f]_1$ w.r.t. $H_{(k)}(B)$ on $\H^{(1)}$ is given by
\[
S\big([f]_1|\!|H_{(k)}(B)\big)= \pi\int_B(1-x^2)f'(x)^2dx - \pi k(k-1)\int_B f(x)^2dx\, ;
\]
hence, the expression on the right-hand side is non-negative, and
\[
S\big([f]_1|\!|H_{(k)}(B)\big) \leq S\big( [f]_1 |\!| {\hob}\big)\, , \ k = 1, 2,\dots
\]
\end{corollary}
\proof
Let $g\in X^{(k)}$ be a smooth function with $g^{(k-1)} = f$ and $g^{(n)}(-1) = 0$, $n = 0,1,\dots, k-2$. By the assumption on the vanishing of the moments $\int_B x^n f(x)dx$, we also have $g^{(n)}(1) = 0$, $n = 0,1,\dots, k-2$. 
Then, by Thm. \ref{thm:ent},
\begin{align*}
S\big([f]_1 |\!|H_{(k)}(B)\big) &= S\big([g]_k |\!| H^{(k)}(B)\big)  \\
&= \pi\int_{B}(1-x^2)g^{(k)}(x)^2dx - \pi k(k-1) \int_{B}g^{(k-1)}(x)^2dx \\
&= \pi\int_{B}(1-x^2)f'(x)^2dx - \pi  k(k-1) \int_{B}f(x)^2dx \, .
\end{align*}
\eproof
The monotonicity
\[
S\big([f]_1 |\!|H_{(k)}(B)\big) \leq S\big([f]_1|\!|H_{(k-1)}(B)\big)
\]
is clear. 
It is instructive to give a direct proof that 
$S\big([f]_k |\! | H^{(k)}(B)\big)$ is non-negative, namely that the right-hand side of \eqref{SkB} is non-negative. 
\begin{proposition}\label{positive}
Let $f\in C^\infty(\bar B)$, then
\[
\int_{B}(1-x^2)f^{(k)}(x)^2dx \geq  k(k-1) \int_{B}f^{(k-1)}(x)^2dx\, .
\]
\end{proposition}
\begin{proof} 
Let us consider the Legendre operator $L=\frac{d}{dx}(1-x^2)\frac{d}{dx}$  on $L^2(B)$ with domain $C^\infty(\bar B)$.  $L$ is essentially selfadjoint on $C^\infty(\bar B)$.  The Legendre polynomials $P_n(x) \equiv \frac1{2^n n!}\frac{d^n}{dx^n}(x^2-1)^n$ form a complete orthogonal family in $L^2(B)$ that diagonalizes $L$
\[
- L P_n = n(n+1) P_n\, ,\ n = 0, 1\dots\, ,
\]
so $n(n+1)$ is a multiplicity one eigenvalue for $-L$ and the spectrum of the closure of $-\bar L$ is $\{n(n+1):n= 0, 1, \dots\}$.

The quadratic form associated with $-L$ is
\[
-(f, Lf) = \int_B(1-x^2)f'(x)^2dx\, ,\ f \in C^\infty(\bar B)\, ,
\]
so we have to show that
\ben\label{ineq}
-( f^{(k-1)} , L f^{(k-1)} ) \geq k(k-1) || f^{(k-1)}||^2\, , \   f \in C^\infty(\bar B)\, .
\een
(with respect to the $L^2(B)$ norm).
Now, $P_n$ is a degree $n$ polynomial, thus $P_n^{(k)} = 0$ if $n = k, k+ 1, \dots$. The spectrum of the restriction $-\bar L$ to the Hilbert subspace $E_k$ generated by $P_k , P_{k+1}, \dots$ has lower bound $k(k+1)$. 
Thus the inequality \eqref{ineq} holds because $ f^{(k-1)}$ belongs to $E_{k-1}$. 
\end{proof}
Let us consider the case $k=2$ in Corollary \ref{cor:ent}. With $f\in S(\RR)$, we have
\[
S\big([f]_1|\!|H_{(2)}(B)\big)= \pi\int_B(1-x^2)f'(x)^2dx - 2\pi \int_B f(x)^2dx\, ,
\]
provided $f$ has zero mean on $B$, namely $\frac12\int_B f(x)dx =0$. For an arbitrary $f\in S(\RR)$, $\bar f \equiv f -  {\frac12}\int_B f(x)dx$ has zero mean and $[\bar f]_1 = [f]_1$, therefore
\begin{align}\label{Sf2}
S\big([f]_1|\!|H_{(2)}(B)\big) &= \pi\int_B(1-x^2)\bar f'(x)^2dx - 2\pi \int_B \bar f(x)^2dx\\
&= \pi\int_B(1-x^2)f'(x)^2dx  - 4\pi {\rm Var}_B(f)\, ,
\end{align}
where
${\rm Var}_B(f) =  \frac12 \int_B f(x)^2dx - \big(\frac12 \int_B f(x)dx\big)^2$ is the variance of $f$ on the probability measure space $(B, dx/2)$. 
One can generalize this expression in the higher derivative case by renormalizing $f$ with higher moments. 

We conclude this section by generalizing Corollary \ref{cor:ent} to the case of the interval $B_R = (-R,R)$, with any $R>0$. 
\bcor\label{cor:ent2}
Let $f\in C_0^{\infty}(\RR)$ with $\int_{-R}^R x^n f(x)dx = 0$, $n = 0, 1, \dots, k-2$. 
The entropy of $[f]_1$ w.r.t. $H_{(k)}(-R,R)$ on $\H^{(1)}$ is given by
\[
S\big([f]_1|\!|H_{(k)}(-R,R)\big)= \frac{\pi}{R} \int_{-R}^R(R^2- x^2)f'(x)^2dx - \frac{\p}{R} k(k-1)\int_{-R}^R f(x)^2dx\, .
\]
\ecor
\proof
Let $\d$ be one-parameter dilation subgroup of $\Mob$, $\d(s) : x\mapsto e^{-s} x$, $x\in\RR$. Then $U^{(1)}\big(\d(s)\big) H^{(k)}(B) = H^{(k)}(B_R)$, $R = e^s$.  Therefore,
\begin{align*}
S\big([f]_1|\!|H_{(k)}(B_R)\big) &= S\big([f]_1|\!| U^{(1)}\big(\d(s)\big) H_{(k)}(B)\big) \\
&= S\big( U^{(1)}\big(\d(-s)\big) [f]_1|\!|  H_{(k)}(B)\big) \\
&= S\big( [f_{s}]_1|\!|  H_{(k)}(B)\big) \, ,
\end{align*}
with $f_s(x) = f(Rx)$. 
By Corollary \ref{cor:ent}, we then have
\begin{align*}
S\big([f]_1|\!|H_{(k)}(B_R)\big) &= \pi\int_B(1-x^2)f_s'(x)^2dx - \pi k(k-1)\int_B f_s(x)^2dx \\
&= \pi R^{2}\int_B(1-x^2)f'(Rx)^2dx - \pi k(k-1)\int_B f(Rx)^2dx \\
&= \frac{\pi}{R} \int_{B_R}(R^2- x^2)f'(x)^2dx - \frac{\pi}{R} k(k-1)\int_{B_R} f(x)^2dx \, .
\end{align*}
\eproof
As a consequence, for every $f\in C_0^\infty(\RR)$, the average entropy converges as $R\to +\infty$
\[
\frac{S\big([f]_1|\!|H_{(k)}(B_R)\big)}R \longrightarrow \pi\int_{-\infty}^{+\infty}f'(x)^{2}dx
\]
 (the average moment normalization vanishes at infinity), and is independent of $k$, as one may expect since the half-line standard subspaces are independent of $k$. 

{

\section{Locality and vector classes}\label{classes}
We now illustrate how our discussion is valid in more generality. 
\subsection{Entropy of a vector class}\label{Evc}
Let $\H$ be a complex space and $H\subset \H$ a standard subspace. The associated \emph{entropy operator} $\E_H$ is defined by the closure of
{
\ben\label{EH2}
\E_H =  A(\Delta_H) + J_H  B(\Delta_H)\, ,
\een
where $A(\l) \equiv - a(\l)\log\l $ and  $B(\l) \equiv b(\l)\log \l$ (with $A(1) = -1$, $B(1) = -1$). More precisely, $\E_H$ is the closure of the right-hand side of \eqref{EH2}. }

Notice that, if $H$ is factorial, then $\E_H$ is the closure of $iP_H i \log\Delta_H$, due to eq. \eqref{PD}. 

$\E_H$ is a real linear, selfadjoint, positive operator (w.r.t. the real part of the scalar product). 

The entropy of a vector $\phi\in\H$ with respect to $H$, defined by \eqref{Sf}, is equivalently given by
\[
S(\phi |\!| H) = \Re(\phi, \E_H  \phi) \, ,
\]
in the quadratic sense, namely
\[
S(\phi |\!| H) = ||\E_H^{1/2} \phi ||^2\, ,
\]
if $\phi$ belongs to the domain of $\E_H^{1/2}$, and $S(\phi |\!| H) = +\infty$ otherwise. 
\bprop\label{locality}
$S(\phi |\!| H) = 0$ iff $\phi\in H'$. That is, $\ker(\E_H^{1/2}) = H'$. 
\eprop
\proof
Of course, $\ker(\E_H^{1/2}) = \ker(\E_H)$ and we first show that $H' \subset \ker(\E_H)$. We begin by noticing that $H'\cap D(\E_H)$ is dense in $H'$, where $D(\E_H)$ is the domain of $\E_H$. Indeed, $E_\e H'\subset H'$ if $E_\e$ is the spectral projection of $\Delta_H$ relative to the interval $(\e^{-1}, \e)$, for all $\e > 1$, and $E_\e \to 1$ as $\e \to +\infty$ strongly (see \cite{LM23}). 

With $\phi'\in H'\cap D(\E_H)$, taking into account that $J_H B(\Delta_H) =  B(\Delta^{-1}_H) J_H$ and $B(\l) = B(\l^{-1})$, we have
{\begin{align*}
\E_H \phi' &=  A(\Delta_H)\phi' + J_H  B(\Delta_H)\phi' \\
&= A(\Delta_H)S_{H'}\phi' +   J_H  B(\Delta_H)\phi'\\
&= A(\Delta_H)\Delta^{1/2}_H J_H\phi' +   J_H  B(\Delta_H)\phi'\\
&= - B(\Delta_H) J_H\phi' +   J_H B(\Delta_H)\phi' = 0 \, ,
\end{align*} }
therefore $\phi' \in \ker(\E_H)$. Since the kernel of a closed linear operator is a closed subspace, we have $H' \subset \ker(\E_H)$.

Conversely, let us show now that $\ker(\E_H) \subset H'$. Equivalently, we have to show that
\[
 i H = (H')^{\perp_\RR} \subset \ker(\E_H)^{\perp_\RR} = \overline{R(\E_H)}\, ,
\]
where $R(\cdot)$ denotes the range. So, it suffices to check that a dense subspace of $H$ is contained in $\overline{R(i\E_H)}$. We first assume that $H$ is factorial, then 
\[
i\E_H |_H  = i \log\Delta_H |_H
\]
is the generator of the one-parameter group of orthogonal operators $\Delta^{it}_H|_H$, so $R(i \log\Delta_H |_H)$ is dense in $H$ because 
$i \log\Delta_H |_H$ is skew-selfadjoint and $\ker(i \log\Delta_H |_H) = \{0\}$. On the other hand, if $H$ is abelian, then $\frac12\E_H$ is the orthogonal projection onto $iH$ \cite{BCD22}, as follows directly from the definition \eqref{EH2}; so $\overline{R(i\E_H)} \supset H$. 

The proposition is now proved because, in general, $H$ is the direct sum of a factorial and an abelian standard subspace (see Sect. \ref{S2}). 
\eproof
\bcor
$S(\phi |\!| H) = S(\phi + \psi|\!| H)$ if $\psi\in H'$. 
\ecor
\proof
Immediate because
\begin{align*}
 S(\phi + \psi|\!| H) &=  \Re(\E_H^{1/2}\phi, \E_H^{1/2}\phi)  + \Re(\E_H^{1/2}\psi, \E_H^{1/2}\psi) + \Re(\E_H^{1/2}\phi, \E_H^{1/2}\psi) + \Re(\E_H^{1/2}\psi, \E_H^{1/2}\phi) \\
 &=  \Re(\E_H^{1/2}\phi, \E_H^{1/2}\phi) = S(\phi |\!| H) 
\end{align*}
is zero if $\psi \in H' = \ker(\E_H^{1/2})$. 
\eproof
We now generalize the notion of entropy of a vector. We shall consider a standard subspace $H$ of a complex Hilbert space $\H$. One can then immediately deal with a general, real linear subspace of $\H$ by projecting onto the standard subspace component of its closure. 

We consider the elements of the quotient space $\H/H'$. Of course, every vector $\phi\in \H$ gives an element $\tilde\phi \equiv \phi +H'$. 
Given $\xi \in \H/H'$, we define the \emph{entropy of the vector class} $\xi$ in $H$ by
\[
S(\xi |\!| H) \equiv S(\phi |\!| H) 
\]
with $\phi\in \H/H'$ any vector such that $\tilde \phi = \xi$. By the locality property, $S(\xi |\!| H)$ is well-defined and strictly positive iff $\xi\neq 0$. 

\subsection{More on the $U(1)$-current and conformal nets}
\label{sectf}
Let $f$ be a real, smooth function on $\bar B$. We choose a real function $g\in S(\RR)$ extending $f$ to $\RR$, then we get a vector $[g]_1\in \H^{(1)}$. If $h$ is another extension of $f$, then $g-h$ is supported in $\RR\setminus B$, thus  $[g]_1 - [h]_1 = [g - h]_1$ 
belongs to $H^{(1)}(B)'$. So
\[
\xi_f \equiv \tilde{[g]_1}
\]
is vector class of $\H^{(1)}$ relative to $H^{(1)}(B)$ that does not depend on the choice of the extension $g$ of $f$. By Cor. \ref{SH1}, we have
\[
S\big(\xi_f |\!| H^{(1)}(B)\big)={\pi}\int_B(1-x^2)f'(x)^2dx\, .
\]
In particular, let $\ell \in C_0^\infty(\RR)$ and $L(x) = \int_{-\infty}^x \ell(t)dt$. Then $\ell$ defines a localized automorphism $\a_\ell$ on the local net of von Neumann algebras on the Bose Fock space over $\H^{(1)}$ \cite{BMT88}. So we have here the formula obtained in \cite{L20}
\[
S(\f\cdot\a^{-1}_\ell |\!| \f)  = S\big(\xi_L |\!| H^{(1)}(B)\big) =  {\pi}\int_B(1-x^2)\ell(x)^2dx\, ,
\]
where $\xi_L$ is the vector class of $\H^{(1)}$, relative to $H^{(1)}(B)$, associated with $L|_{\bar B}$,
and the left term is Araki's relative entropy between the vacuum state $\f$ and the charged state $\f\cdot\a^{-1}_\ell$ on the local von Neumann algebra associated with $B$ in second quantization. 

We may also consider the entropy of $\xi_f$ relative to $H_{(2)}$
\[
S\big(\xi_f |\!| H_{(2)}(B)\big) \equiv S\big([g]_1 |\!| H_{(2)}(B)\big) \, ,
\]
that, again, does not depend on the choice of the extension $g$ of $f$. With $\ell$ and $L$ as above, we have
\[
S\big(\xi_L |\!| H_{(2)}(B)\big) =  \pi\int_B(1-x^2)\ell(x)^2dx  - 4\pi {\rm Var}_B(L) \, ,
\]
so the entropy increase of the localized automorphism $\a_\ell$ on $B$ is given by the variance of $L$ on $B$:
\[
S\big(\xi_L |\!| H_{(1)}(B)\big) - S\big(\xi_L |\!| H_{(2)}(B)\big) =  4\pi {\rm Var}_B(L) \,  . 
\]
We end this section by showing an analog of the strict inclusion \eqref{sincl} in the conformal case. Let $\A$ be a non-trivial diffeomorphism covariant, local conformal net on $S^1$ on a Hilbert space $\H$. In particular, $\A$ is a local, Möbius covariant net on $S^1$, and we consider the associated local, Möbius covariant net $H$ of standard subspace on $\RR$:
\[
H_\A(I) \equiv \overline{\A(I)_{\rm sa}\Omega}\, , \quad I\in \I_0\, ,
\]
with $\Omega$ the vacuum vector. 
\bprop
With $\A$ as above, the net $H_\A$ is not a dual net, namely $H_\A(I) \nsubseteq H_\A^d(I)$, $I\in \I_0$. 
\eprop
\proof
Let $U$ be the covariance unitary, projective representation of the diffeomorphism group $\Diff(S^1)$ associated with $\A$. The restriction $U|_\Mob$ is the direct sum of irreducible, positive energy, unitary representations of $\Mob$, $U = \bigoplus_k^\infty n_k U^{(k)}$, with multiplicity $n_k$. Representations $U^{(k)}$ with $k>2$ appear in this direct sum with multiplicity $n_k \geq 1$ (see e.g. \cite{KRR13}). 
$H$ is thus the direct sum 
\[
H_\A = \bigoplus_k^\infty n_k H^{(k)}
\]
of the nets $H^{(k)}$. The proposition is proved because $H^{(k)}$ is a dual net iff $k=1$. 
\eproof
Note that if $\A$ is rational, then Haag duality on $\RR$ holds for $\A$ \cite[Sect. 5]{LX04} (the split property holds by \cite{MTW18}). 

\medskip

\noindent
{\bf Acknowledgements.} 
The authors acknowledge the Excellence Project 2023-2027 MatMod@TOV awarded to the Department of Mathematics, University of Rome Tor Vergata. {They} are supported by GNAMPA-INdAM.
R.L.  {thanks  D. Buchholz and K.H. Rehren for} stimulating conversations during his visit to Göttingen in November 2021.


\begin{thebibliography}{99}

\bibitem{ArBook} 
{\sc H. Araki},
``Mathematical Theory of Quantum Fields'',
Internat. Ser. Monogr. Phys. 101, Oxford Univ. Press, Oxford 2009. 

\bibitem{Ar63} 
{\sc H. Araki},
\emph{A lattice of von Neumann Algebras associated with the Quantum Theory of a Free Bose Field},
J. Math. Phys. 4, 1343--1362 (1963). 

\bibitem{Ar} 
{\sc H. Araki},
\emph{Relative entropy of states of von Neumann algebras},
Publ. RIMS Kyoto Univ. 11, 809--833, (1976). 

\bibitem{BCD22}  
{\sc  H.  Bostelmann, D.  Cadamuro, S.  Del Vecchio},
\emph{Relative entropy of coherent states on general CCR algebras},
Comm. Math. Phys. 389, 661--691 (2022). 

\bibitem{BGL94} {\sc  R. Brunetti, D.  Guido, R  Longo}, 
{\it Group cohomology, modular theory and space-time symmetries}, 
Rev. Math. Phys. 7, 57--71 (1994).

\bibitem{BGL02} {\sc  R. Brunetti, D.  Guido, R  Longo}, 
{\it Modular localization and Wigner particles},
Rev. Math. Phys. 14, 759--786, (2002).

\bibitem{BF23}
{\sc D. Buchholz, K. Fredenhagen}, 
{\it Arrow of time and Quantum Physics},
Found. Phys. 53, 85 (2023). 

\bibitem{BMT88} 
{\sc D. Buchholz, G. Mack,  I. Todorov}, 
{\it The current algebra on the circle as a germ of local field theories},
Nucl. Phys. B, Proc. Suppl. 56, 20 (1988). 

\bibitem{BSM90} 
{\sc D. Buchholz, H.  Schulz-Mirbach}, 
{\it Haag duality in conformal quantum field theory},
Rev. Math. Phys. 2, 105 (1990). 

\bibitem{CLRR22} {\sc F. Ciolli, R. Longo, A. Ranallo, G. Ruzzi}, 
\emph{Relative entropy and curved spacetimes},
J. Geom. Phys. 172 104416 (2022). 

\bibitem{CLR20} {\sc F. Ciolli, R. Longo, G. Ruzzi}, 
\emph{The information in a wave}, 
Comm. Math. Phys. 379, 979--1000, (2020).

\bibitem{GP}
{\sc A.~Garbarz, G.~Palau},
{\it Relative entropy of an interval for a massless boson at finite temperature},
Phys. Rev. D 107, 125016 (2023).

\bibitem{GL95}
{\sc  D. Guido, R. Longo}, 
{\it An algebraic spin and statistics theorem}, 
Comm. Math. Phys.  172, 517-- 533 (1995). 

\bibitem{GL96}
{\sc  D. Guido, R. Longo}, 
{\it The conformal spin and statistics theorem}, 
Comm. Math. Phys. 181, 11--35,  (1996). 

\bibitem{GLW98}
{\sc  D. Guido, R. Longo, H.-W. Wiesbrock},
{\it Extensions of conformal nets and superselection structures},
Comm. Math. Phys. 192, 217--244 (1998). 

\bibitem{KRR13}
{\sc V.G. Kac, A.K. Raina, N. Rozhkovskaya},
``Bombay lectures on highest weight representations of infinite dimensional Lie algebras'',
Second edition,
Adv. Ser. Math. Phys., 29, World Scientific Publishing, Hackensack, NJ, 2013.

\bibitem{L}  
{\sc R. Longo}, 
{\it Real Hilbert subspaces, modular theory, $SL(2, R)$ and CFT}, 
in: ``Von Neumann algebras in Sibiu'',  Theta Ser. Adv. Math., 10, 33--91, Theta, Bucharest, (2008).

\bibitem{L20} 
{\sc R. Longo}, 
{\it Entropy distribution of localised states}, 
Comm. Math. Phys. 373, 473--505 (2020).  

\bibitem{L19} 
{\sc R. Longo}, 
{\it Entropy of coherent excitations}, 
Lett. Math. Phys. 109,  258--260 (2019). 

\bibitem{L23} 
{\sc R. Longo}, 
{\it Signal communication and modular theory}, 
Comm. Math. Phys. 403, 473--494 (2023). 

\bibitem{LMR16} 
{\sc R. Longo, V. Morinelli, K.-H. Rehren}, 
{\it Where infinite spin particles are localizable},
Comm. Math. Phys.  345, 587--614 (2016).  

\bibitem{LM23} 
{\sc R. Longo, G. Morsella},
{\it The massless modular Hamiltonian}, 
Comm. Math. Phys. 400(2), 1181--1201 (2023). 

\bibitem{LX04} 
{\sc R. Longo, F. Xu},
{\it Topological sectors and a dichotomy in Conformal Field Theory}, 
Comm. Math. Phys. 251, 321--364 (2004). 

\bibitem{M16} 
{\sc V. Morinelli}, 
{\it The Bisognano–Wichmann property on nets of standard subspaces, some sufficient conditions},
Ann. Henri Poincaré 19, 937--958 (2018). 

\bibitem{MTW18} 
{\sc V. Morinelli, Y. Tanimoto, M. Weiner}, 
{\it Conformal covariance and the split property},
Comm. Math. Phys. 357 (2018), no. 1, 379--406.

\end{thebibliography}
\end{document}